\documentclass[11pt]{amsart} 

\usepackage{amsmath,amstext,amsfonts,amssymb,url} 

\newtheorem{theorem}{\bf Theorem}[section] 
\newtheorem{corol}[theorem]{\bf Corollary} 
 
\newtheorem{lemma}[theorem]{\bf Lemma} 
\newtheorem{prop}[theorem]{\bf Proposition}

\DeclareMathOperator{\Aut}{Aut} 
\newcommand{\abs}[1]{\lvert#1\rvert} 
\newcommand{\noi}{{\noindent}} 
 
\newcommand{\ctl}{\centerline} 
 
\newcommand{\la}{{\langle}} \newcommand{\ra}{{\rangle}} 
  
\newcommand{\nd}{{\text{ and }}}

\newcommand{\stx}{\begin{smallmatrix}} 
\newcommand{\estx}{\end{smallmatrix}} 

\newcommand{\Lb}{{\Lambda}}

\newcommand{\Z}{{\mathbb Z}} 

\newcommand{\cB}{{\mathcal B}} 
\newcommand{\cC}{{\mathcal C}}

\newcommand{\cS}{{\mathcal S}} 

\DeclareMathOperator{\GL}{GL}
\newcommand{\sn}{\cS^n} 
\newcommand{\glnz}{\GL_n(\Z)}

\begin{document}

\title{Bases of minimal vectors in lattices, III} 
\author{Jacques Martinet} 
\address{%
Institut de Math\'ematiques\\ 
351, cours de la Lib\'eration\\ 
33405 Talence cedex\\ 
France} 
\email{Jacques.Martinet@math.u-bordeaux1.fr} 
\author{Achill Sch\"urmann} 
\address{%
Institute of Mathematics\\
University of Rostock\\
18051 Rostock\\
Germany}
\email{achill.schuermann@uni-rostock.de}

\begin{abstract} 
We prove that all Euclidean lattices of dimension~$n\le 9$ 
which are generated by their minimal vectors, also possess a basis 
of minimal vectors. By providing a new counterexample, we show
that this is not the case for all dimensions~$n\ge 10$.
\end{abstract} 

\thanks{The second author was supported by the
Deutsche Forschungsgemeinschaft (DFG) under 
grant SCHU 1503/4-2 and by the Universit\'e Bordeaux 1.} 
\maketitle

\section{Introduction}\label{secintroit}

In their paper \cite{cs-1995}, Conway and Sloane constructed an example of an 
$11$-dimensional lattice generated by its minimal vectors, but having 
no basis of minimal vectors. They left open the question 
of the existence of such lattices in lower dimensions. 
In \cite{martinet-2007}, the first author proved that such lattices do not exist
in dimensions $n\leq 8$, leaving open their existence in dimension~$9$ and~$10$. 
In this paper we fully resolve this question.

\begin{theorem}\label{mainthm} 
A lattice of dimension $n\le 9$ which is generated 
by its minimal vectors, has also a basis of minimal vectors. 
In all dimensions $n\ge 10$ there exist lattices which are generated 
by their minimal vectors, but have no basis of minimal vectors. 
\end{theorem}

For standard terminology on lattices used here and in the sequel
we refer the reader to~\cite{martinet-2003}. 
Given, an $n$-dimensional Euclidean vector space~$E$,
we say that a lattice $\Lb\subset E$ is {\em well rounded} 
if its minimal vectors span~$E$. Any system of $n$~independent 
minimal vectors then generates a sublattice $\Lb'$ 
of finite index in~$\Lb$, referred to as {\em Minkowskian sublattice}.
We denote by $\imath=\imath(\Lb)$ 
the {\em maximal index} $[\Lb:\Lb']$ for Minkowskian sublattices~$\Lb'$ of~$\Lb$.

Our proof of Theorem~\ref{mainthm} 
makes use of knowledge about possible values 
of $\imath$ and of the corresponding structures of the set 
of minimal vectors of $\Lb$.
Our basic references for this information are~\cite{martinet-2001} 
(in particular Table~11.1) and \cite{kms-2010} (in particular Tables~2 to~10),
which extend previous works of Watson, Ryshkov, and Zahareva 
up to dimension~$9$.
We give a brief sketch of the used results 
and recall some notations for this paper in Section~\ref{sec:first_paper}.

We split the proof of Theorem~\ref{mainthm} 
into different cases according to the maximal
index~$\imath$ of a putative counterexample.
Loosely speaking, proofs are straightforward for lattices having 
a small or a large maximal index. 
In dimension~$9$ ``large index'' means $\imath\ge 10$: by the results in~\cite{kms-2010},
there exist sufficiently many minimal vectors in these cases, 
from which bases can be extracted. 
Some more details are given in Section~\ref{seclargeindex},
where we also treat $9$-dimensional lattices with maximal
index~$\imath=7,8,9$. 
A ``small index'' in dimension~$9$ 
means $\imath \le 4$, a case that we consider in Section~\ref{sec:2elementary}.

This leaves us with the two more difficult cases of $\imath=5,6$,
for which we use some computer assistance. The difficult case with $n=9$ and $\imath=6$
is treated in Section~\ref{sec:index6}. 
In Section~\ref{sec:index5}, we reduce the other difficult case with 
$n=9$ and $\imath=5$ to the study of just one special type of
lattices. It turns out that such lattices do not exist for $n=9$.
However, the study of corresponding lattices for $n=10$, lead
us to the counterexamples described in Section~\ref{sec:counterexample},
by which we finish the proof of Theorem~\ref{mainthm}.
In Section~\ref{sec:algorithmic_approach} we not only give some
background information on the computer calculations, but we also
explain how the same techniques could be used for a general
algorithmic approach to the proof of Theorem~\ref{mainthm}.
It should be noted that such a fully computerized proof seems
practically infeasible and that only the interplay of human reasoning
and computer assistance allowed us to obtain the dimension~$9$ part of the theorem.

\section{Classification of minimal classes}  \label{sec:first_paper}

This paper relies largely on the results in~\cite{kms-2010}. 
There we describe all the $\Z/d\Z$-codes arising from a pair 
$(\Lb,\Lb')$ of $9$-dimensional lattices, such that $\Lb$ is generated 
by its set $S=S(\Lb)$ of minimal vectors and $\Lb'$ has a basis 
$\cB=(e_1,\dots,e_9)$ of vectors of~$S$; here, $d$ is the annihilator 
of $\Lb/\Lb'$, and the code words are the elements 
$(a_1,\dots,a_9)\in(\Z/d\Z)^9$ 
such that $\frac{a_1 e_1+\dots+a_9 e_9}d$ belongs to~$\Lb$.
We classify pairs $(\Lb,\Lb')$ according to possible
structures of $\Lb / \Lb'$, merely viewed as an abstract Abelian group.
Using the standard convention for quoting Abelian groups by their
elementary divisors, we speak for example 
of type~$(3,2)$ for groups of order~$6$ isomorphic to $\Z / 3\Z \times \Z / 2\Z$.
In \cite{kms-2010} we show that 
quotients $\Lb/\Lb'$ may have any possible structure of order up to index~$10$, 
or they may be of type $(12)$, $(6,2)$, $(4,4)$, $(4,2,2)$ or $(2,2,2,2)$ for larger index.

Given a code~$C$, we attach in \cite{kms-2010} a unique {\em minimal class}~$\cC_C$ 
(in the sense of \cite[Section~9.1]{martinet-2007}; cf. Proposition~3.1 in \cite{kms-2010}).
This is an equivalence class of lattices with the property that 
for any pair $(\Lb,\Lb')$ defining $C$, 
the set $S(\Lb)$ contains $S(\Lb_0)$ for some lattice $\Lb_0\in\cC_C$. 
In other words, the sets of minimal vectors of lattices in $\cC_C$, 
which are well-defined up to $\glnz$ equivalence, are minimal with respect to inclusion.

The general method that we can use to prove the existence or
nonexistence of a type, 
is to show that the set of lattices realizing a given code or minimal
class is non-empty or empty.
Instead of working with lattices directly, we work with the space of Gram matrices of lattice bases,
that is, with the space of positive definite, real symmetric matrices.
The question of existence or nonexistence of a lattice type 
can be decided based on polyhedral computations
with rational coordinates, hence, with 
linear algebra that can be rigorously checked using a computer.
For a detailed description we refer to~\cite{kms-2010}.

\section{A general algorithmic approach}  \label{sec:algorithmic_approach}

For the proof of Theorem~\ref{mainthm} presented here
we use computer assistance to rule out 
two difficult cases in dimension~$9$ in Sections~\ref{sec:index6} and~\ref{sec:index5}.
The used techniques 
are general enough, however, to allow (in
principle) a proof of Theorem~\ref{mainthm} entirely based on computer reasoning. 
In dimension~$n$ the general algorithmic approach can be split into two
main components:

\medskip

\begin{enumerate}
\item For each minimal class of well rounded~$n$-dimensional lattices obtain a
  polyhedral realization space~$P$ of Gram matrices.
\item For each face of~$P$ obtain a ``typical Gram matrix'' 
         and check if its minimal vectors generate but do not
         provide a basis of~$\Z^n$.
\end{enumerate}

\medskip

For $n=9$, the minimal classes of well rounded lattices are classified in~\cite{kms-2010}.
There, the non-existence or existence of a minimal class is
decided by checking if a corresponding polyhedral realization space
is empty or not. Thus Step~1 is carried out in~\cite{kms-2010}.
Typical Gram matrices 
in {\tt PARI/GP} format \cite{pari} 
and additional information   
can be obtained from the file {\tt Gramindex.gp} in the ``online appendix'' 
of~\cite{kms-2010}.

Below we explain Step~2. 
For an implementation we used {\tt MAGMA}~\cite{magma}, 
together with the program {\tt lrs}~\cite{lrs} to deal with the
necessary polyhedral computations. 
For the exclusion of specific cases in this paper, we
use Step 2 only for a few specific cases: 22 faces in
Sections~\ref{sec:index6}
and one face in Section~\ref{sec:index5}.

\subsection*{Polyhedral realization spaces of minimal classes}

A minimal class~$\cC_C$ has a polyhedral ``realization space''
of Gram matrices attached to it. To see this connection, assume the code $C$ 
contains $k$~code words $a^{(i)}$, $i=1,\ldots,k$, the lattice $\Lb'$ has a basis of minimal 
vectors $e_1,\ldots, e_n$ of $\Lb$ and 
$\Lb = \langle \Lb', f_1,\dots,f_k \rangle$ with 
$$
f_i=\frac{a_1^{(i)}e_1+\dots+a_n^{(i)}e_n}d \, ,
$$ 
for $i=1,\ldots,k$. 
With respect to a fixed chosen basis $B=(b_1,\ldots, b_n)$ of $\Lb$,
the $e_i$ have coordinates $\bar{e}^{(i)}\in \Z^n$, which can be expressed solely in terms 
of the $a_j^{(i)}$ and~$d$. Note that these coordinates are completely independent 
of the specific lattices~$\Lb$ and~$\Lb'$.

{\em Example.}
Let us look at a specific case that we consider in Section~\ref{sec:index5}:
Let $n=9$, $d=5$ and $\Lb = \langle \Lb', e \rangle$ with $e$ as in~\eqref{eqn:edef}.
Then we can choose $B=(e,e_2,\ldots,e_9)$ as a basis for~$\Lambda$
and the coordinate vector $\bar{e}^{(1)}$ with respect to $B$ is $(5,-1,-1,-2,-2,-2,-2,-2,0)$.

\bigskip

Assuming the minimum of $\Lb$ is $1$, we know that 
the Gram matrix of $B$ is contained in the affine subspace 
\begin{equation} \label{eqn:linearspace}
\{
G \in \sn \; | \; G[\bar{e}^{(i)}] = 1 \mbox{ for } i=1,\ldots,n
\}
\end{equation}
within the space $\sn$ of real symmetric $n\times n$ matrices.
Here, we make use of the notation $G[a]$ for $a^t G a$.
The infinitely many linear conditions $G[a]\geq 1$ on $\sn$ that are satisfied 
for all non-zero integral vectors~$a$, 
define a set referred to as {\em Ryshkov polyhedron}.
It is a locally finite polyhedral set (see~\cite[Chapter~3]{schuermann-2008} for details).
Its intersection with~\eqref{eqn:linearspace} is either empty if the code can not be realized,
or it is a polytope~$P$ (convex hull of finitely many Gram matrices) if the code can be realized.
In the latter case, the relative interior points (within the affine subspace spanned by the polytope)
are Gram matrices from bases of lattices in the minimal class~$\cC_C$.
Note that all of these lattices have the same set of minimal vectors,
so that it is sufficient to know one relative
interior Gram matrix, which can be considered ``typical'' for its class.

\subsection*{Faces and typical Gram matrices}

For a given minimal class, the boundary of the constructed
polytope~$P$ 
(with respect to the topology of the affine subspace spanned by it)
is subdivided into {\em faces}, that is, 
into parts which themselves are polytopes of lower dimension.
Each face is uniquely defined by some additional
affine equations~$G[a^{(j)}]=1$, with coordinate vectors~$a^{(j)}\in \Z^n$.
These coordinate vectors are the same for all relative interior Gram
matrices of a face and there are only finitely many of them.
In terms of corresponding, they give coordinates of additional minimal vectors 
with respect to the chosen basis.

By considering a typical (any relative interior) Gram matrix for each
face of~$P$, one can decide for the corresponding minimal
class, whether or not there exist lattices in the class which are generated by minimal
vectors but do not provide a basis among them: For each typical Gram
matrix~$G$ one has to check whether or not the coordinate vectors 
$\{ a\in \Z^n \; | \; G[a] = 1\}$ attaining the minimum generate $\Z^n$ but
do not provide a basis for $\Z^n$.

\subsection*{Choosing specific faces}

For a full automated proof of Theorem~\ref{mainthm}, the number of
cases to be considered would be huge. We use computer
assisted checks only for specific cases.
In each considered case, 
the coordinate vectors~$a^{(j)}$ and $\bar{e}^{(i)}$ of assumed 
minimal vectors generate but do not provide a basis of~$\Z^9$.

{\em Example. }
In the case with $d=5$ considered in Section~\ref{sec:index5},
we assume the existence of an additional minimal vector ($x$, given in~\eqref{eqn:xdef})
which has coordinates $a^{(1)}=(-4,0,0,2,2,-1,-1,-1,-1)$ with respect
to the basis $B=(e,e_2,\ldots,e_9)$. 
Here, $a^{(1)}$ together with 
$\bar{e}^{(i)}$, for $i=1,\ldots, 9$, generate $\Z^9$, but no choice
of nine of the ten vectors gives a basis for~$\Z^9$. 

\medskip

Often, the linear conditions we start with imply additional linear
conditions~$G[a^{(j)}]=1$ with
additional coordinate vectors~$a^{(j)}\in \Z^9$;
or in the language of lattices: the existence of some minimal vectors implies the existence of others. 
A full list of all implied coordinate vectors can be determined by computing the minimal vectors
of a typical Gram matrix. If this full list contains a basis of~$\Z^9$,
then so do the minimal vectors for all lattices of the considered type. 
Note that all of these rational linear algebra operations 
can rigorously be verified using a computer.

{\em Example.}
In the $n=9$, $d=5$ case of Section~\ref{sec:index5} we find $10$~pairs of 
additional minimal vectors from the Gram matrix in~\eqref{eqn:n9d5gram}.
The full list of minimal coordinate vectors obtained in this way contains a basis of~$\Z^9$, 
excluding this case as a counterexample in dimension~$9$.

\medskip

\subsection*{Exploiting polyhedral symmetries}

In higher dimensions, i.e. for $n\geq 9$, 
the polyhedral computations necessary to decide 
whether or not a code is realizable can be quite involved. 
In these cases we can try to exploit available symmetries
to make the computations feasible.
The automorphism group of the code, yields an automorphism group of the corresponding minimal class:
\begin{equation} \label{eqn:autogrp}
\Aut \cC_C =
\{
U\in\glnz
\; | \;
U\bar{e}^{(i)} \in \{ \bar{e}^{(1)},\ldots, \bar{e}^{(n)} \} 
\mbox{ for all } i = 1,\ldots, n
\}
\end{equation}
The polytope $P \subset \sn$ described above is invariant with respect to this group:
We have $U^t P U = P$ for all $U\in \Aut \cC_C$. 
The same is true for the set of vertices of~$P$
and therefore the vertex barycenter of $P$ (if non-empty)
is contained in the linear subspace
\begin{equation} \label{eqn:invariantsubspace}
\{
G\in \sn \; | \;
U^t G U = G \mbox{ for all } U \in \glnz
\}
\end{equation}
of $\Aut \cC_C$-invariant Gram matrices. 
Thus for checking feasibility of a given code~$C$
we can restrict the polyhedral computations and search for an interior point
within the linear subspace~\eqref{eqn:invariantsubspace}.

Moreover it is also possible to make use of symmetries when considering 
a fixed code~$\cC_C$, together with some coordinate vectors $a^{(j)}$.
Instead of restricting to the invariant linear subspace~\eqref{eqn:invariantsubspace} 
coming from the symmetry group~$\Aut \cC_C$,
we can restrict to a corresponding linear subspace obtained from a subgroup of~$\Aut \cC_C$, 
for which also the set of coordinate vectors~$a^{(j)}$ is preserved.

{\em Example.}
In our example with $n=9$, $d=5$, 
the coordinate vectors~$\bar{e}^{(i)}$
and~$a^{(1)}=(-4,0,0,2,2,-1,-1,-1,-1)$ 
are for example invariant with respect to
permutations of~$e_2,e_3$, of~$e_4,e_5$ and of~$e_6,e_7,e_8$.

\section{Lattices of large index}\label{seclargeindex}

For $n=9$ and ``very large'' indices, namely for $\imath(\Lb)\ge 10$, 
lattices are generated by minimal vectors, and a basis of minimal 
vectors is ``almost'' in evidence on typical Gram matrices from the online appendix 
of~\cite{kms-2010}. Indeed, most of the matrices turn out to 
have diagonal entries equal to the minimum of the lattice 
we consider, and in one case where a diagonal entry was larger 
than this minimum, we can easily conclude by listing all minimal 
vectors of the lattice.

\smallskip

To work with $\imath=7,8,9$ is less simple: in these cases,
it may happen that the lattices $\Lb$ of a minimal class $\cC$ are not generated by their minimal vectors.
In fact, in some cases, the only minimal vectors are the nine $e_i$ spanning the sublattice $\Lb'$.
Or it may happen that all of the additional minimal vectors do not generate~$\Lb$,
for example if they all lie in $\Lb'$. 
We must in these cases explicitly use the existence of extra minimal vectors.

\smallskip 

The following proposition will be used (at least implicitly) 
from this Section onwards.

\begin{prop}\label{propbound} 
Let $(\Lb,\Lb')$ be a pair of $n$-dimensional lattices, where
$\Lambda'$ is generated by minimal vectors $(e_1, \ldots, e_n)$ of~$\Lb$.
Let $d$ the annihilator of $\Lb/\Lb'$ and let 
$x=\frac{a_1 e_1+\dots+a_n e_n}{d'}$ with $a_i, d'\in\Z$ and $d'\mid d$. 
Then the absolute values of the $a_i$ are bounded from above 
by $\abs{d'}$. 
\end{prop} 

\begin{proof}
We may suppose that $d'$ and the $a_i$ are coprime. 
Let $L=\la\Lb',x\ra$ contain $\Lb'$ to index~$d'$. 
Let $i$ such that $a_i\ne 0$. We have 
$$e_i=\frac{-d' x-\sum_{j\ne i}\,a_j e_j}{a_i}\,,$$ 
which shows that $L$ contains to index~$\abs{a_i}$ 
the lattice $M$ generated by $x$ and the $e_j,\,j\ne i$. 
We have $[\Lb:M]=[\Lb:L]\cdot[L:M]\le\imath$, hence 
$$\abs{a_i}=[L:M]\le\frac{\imath}{[\Lb:L]}=d'\,.$$ 
\end{proof}

\smallskip

\subsection*{Proof of Theorem 1.1 for cyclic quotients of order $\mathbf{d=9,8,7}$}
We first consider cyclic quotients of order $d=9,8,7$, 
writing $\Lb=\la\Lb',e\ra$ for some vector $e$ of the form 
$e=\dfrac 1d\,(\sum_{i=1}^9a_i e_i)$ with 
$a_i\in\Z$. Since quotients $\Z/d\Z$ do not exist in dimension~$8$, 
all $a_i$ are non-zero and we may choose them modulo~$d$.
For every integral $c$ prime to $d$ we have $\Lb=\la\Lb', c e\ra$,
allowing us to choose all $a_i$ in $\{1,2,\dots,\frac d2\}$.
The $\Z/d\Z$-code generated by the word 
$(a_1,\dots,a_{d/2})$ is well defined by this sequence 
up to permutation, that is, it is defined by the numbers $m_i$ of coefficients $a_j$ 
equal to~$i$. The transformation $e\mapsto c e$ induces an action 
of $(\Z/d\Z)^\times/\{\pm 1\}$ 
which amounts to a circular permutation 
of $(m_1,m_2,m_3)$ if $d=7$, of $(m_1,m_2,m_4)$ if $d=9$, 
and the exchange $m_1\leftrightarrow m_3$ if $d=8$.

Since $7,8$ and $9$ are prime powers, the hypothesis 
``$\Lb$ is generated by its minimal vectors'' 
amounts to the existence of a minimal vector $x\in\Lb$ 
which generates~$\Lb$ modulo $\Lb'$. 
In \cite[Table~2]{kms-2010} the possible codes are listed, 
up to a permutation as above, which we have to take into account here. 
For instance, if $\Lb$ is constructed using the code modulo~$7$ 
for which $(m_1,m_2,m_3)=(4,2,3)$, 
the hypothesis ``$\Lb$ is generated by its minimal vectors'' implies the existence of a minimal vector 
\hbox{$x\in e+\Lb'$} for an $e$ associated with any of the three systems 
$(m_1,m_2,m_3)$ $=$ $(4,2,3)$, $(3,4,2)$ and $(2,3,4)$, 
and we must thus consider three cases for one code listed 
in~\cite{kms-2010}.

\medskip 

\noi\underbar{$d=9$}. 
Six codes are listed in \cite[Table~2]{kms-2010}. An inspection of the 
corresponding Gram matrices
shows that a basis of minimal vectors 
exists for the first five, and that the minimal vectors generate 
a sublattice of index~$3$ in the remaining case with $(m_1,m_2,m_3,m_4)=(2,2,3,2)$. 
So we must assume the existence of at least one additional minimal vector here.
As we may permute $m_1, m_2$ and $m_4$ in this case,
we may assume that $\Lb$ contains a minimal vector $x=\frac{a_1 e_1+\dots+ a_9 e_9}9$ 
with $a_1\equiv a_2\equiv 1\mod 9$.
We have $\abs{a_i}\le\imath= 9$ (because $\imath(\Lb)\le 9$), 
hence $a_1,a_2=1$ or $-8$. If $a_1=a_2=-8$, then we may write 
$e_1+e_2+x=\frac{-x+a_3 e_3+\dots+a_9 e_9}8$, constructing this way 
a lattice of index $8$ in dimension~$8$, a contradiction. 
Hence $a_1=1$, say, and $(x,e_2,\dots,e_9)$ is a basis of minimal 
vectors for $\Lb$. 

\smallskip 

\noi\underbar{$d=8$}. For $14$ out of the $19$ codes listed 
in \cite[Table~2]{kms-2010}, Gram matrices show the existence of a basis
of minimal vectors. In the remaining $5$~cases, $S(\Lb)$ generates 
a lattice of index~$8$ 
(for systems $(3,4,2,0)$, $(3,3,2,1)$ and $(3,2,2,2)$) 
or~$2$ (for systems $(2,4,2,1)$ and $(3,1,3,2)$), which we now 
consider together with the allowed permutation~$(1,3)$. 

In all cases, we have $m_1\ge 1$, and the argument used 
for denominator~$9$ still works: 
for $x$ minimal in $e+\Lb'$, we can exclude that two coefficients 
are equal to~$-7$ (because we would construct in this way 
an $8$-dimensional lattice with $\imath=7$), so that we may assume 
that, say, $a_1=1$, and $(x,e_2,\dots,e_9)$ is then a basis 
of minimal vectors. 

\smallskip 

\noi\underbar{$d=7$}. For six out of eight systems $(m_1,m_2,m_3)$, 
Gram matrices show the existence of a basis of minimal vectors 
for $\Lb$. 
In the remaining two systems, we have $S(\Lb)=S(\Lb')$, 
and we must use the hypothesis ``$e+\Lb'$ contains some minimal 
vector~$x\,$'' for all circular permutations of these two systems, namely 
$$(5,2,2),\,(2,5,2),\,(2,2,5),\,(4,2,3),\,(3,4,2),\,(2,3,4)\,.$$ 
Write $e=\frac{a_1 e_1+\dots + a_9 e_9}7$ and 
$x=\frac{b_1 e_1+\dots + b_9 e_9}7$, with $b_i\equiv a_i\!\!\mod 7$. 
Since $\imath(\Lb)=7$, we have $b_i=a_i$ or $b_i=-(7-a_i)$. 
Denoting by $m'_i$ ($i=1,\dots,6$) the number of subscripts $j$ 
such that $\abs{b_j}=i$, 
we have $m'_1+m'_6=m_1$, $m'_2+m'_5=m_2$ and $m'_3+m'_4=m_3$,  
and $x+\sum_{b_i=-6}\,e_i=\frac{-x+\sum_{bi\ne -6}\, b_i e_i}6$. 
If $m'_1\ge 1$, say, $b_1=1$, then $(x,e_2,\dots,e_9)$ is a basis 
of minimal vectors, so that we may assume that $m'_6=m_1$. 
By the equality above, there exist lattices $L,L'$ 
with $[L:L']=6$ in dimension $n'=10-m_1$, which is possible only if 
$10-m_1\ge 8$, i.e., $m_1=2$, and then the corresponding system 
$(M_1,M_2,M_3)$, namely $(1+m'_5,m'_2+m'_4,m'_3)$ must be one 
of the six systems listed in \cite{martinet-2001}. For five out of these six 
systems, Section~9 of \cite{martinet-2001} immediately shows that there exists 
a basis of minimal vectors for $L$, hence also for~$\Lb$. 
We may thus assume that $(M_1,M_2,M_3)=(3,3,2)$, 
hence first that $m'_5=m_1=2$, next that $m'_2+m'_4=3$, 
so that 
$$x=
\frac{-6(e_1+e_2)-5(e_3+e_4)+3(e_5+e_6)+b_7 e_7+b_8 e_8+b_9 e_9}7$$ 
with $b_7,b_8,b_9=2$ or~$-4$. 
We now write the equality above in the form 
{\small 
$$2(x+e_1+e_2+e_3+e_4)-(e_5+e_6)\pm e_7\pm e_8\pm e_9= 
\frac{-x+e_3+e_4\pm e_7\pm e_8\pm e_9}3\,,$$} 
\hskip-.1cm 
which shows that $y=\frac{-x+e_3-2e_4\pm e_7\pm e_8\pm e_9}3$ 
is minimal. Replacing $x$ by its components on the $e_i$, 
we obtain 
$$y=\frac{2(e_1+e_2)+4e_3-3e_4+e_5+e_6\pm e_7\pm e_8\pm e_9}7\,,$$ 
which shows that $(e_1,\dots,e_8,y)$ is a basis of minimal vectors 
for a lattice containing~$\Lb'$ to index $7$, hence 
equal to~$\Lb$. 

\medskip

\subsection*{Proof of Theorem 1.1 for non-cyclic quotients of order $\mathbf{d=9,8}$}

We now turn to non-cyclic quotients, which are of one of the types 
$(3,3)$, $(4,2)$ or $(2,2,2)$. 

\smallskip 

\noi\underbar{Type $(3,3)$}. \cite[Table~6]{kms-2010} shows that $\Lb$ 
is of the form $\la\Lb',e,f\ra$ where $e,f$ have denominator~$3$, 
where $e$ (resp. $f$) has $6$ (resp. $6$, $6$ or $7$) 
non-zero components. This shows that the $e-e_i$, $i\le 6$ 
are minimal, as well as $6$ vectors $f\mp e_j$ in the first 
two cases, so that we obtain a basis of minimal vectors for $\Lb$ 
by replacing two convenient vectors $e_k,e_\ell$ by vectors of the form 
$e-e_i$, $f-e_\ell$. In the third case, we use the fact that there
exists some minimal vector $y$ in one of the cosets of $f$, $f+e$ 
or $f-e$ modulo $\Lb'$, say, $y\in f+\Lb'$. 
(The automorphism group of the code exchanges these three cosets.) 
We now consider $\Lb''=\la\Lb',f\ra$. We have $\imath(\Lb'')=3$ 
(because $[\Lb:\Lb'']=3$) hence a minimal vector $y\in f+\Lb'$ must have 
$7$~odd components equal to $\pm 1$ or $\pm 2$, not all equal 
to $\pm 2$ (as in the proof of Lemma~3.1 in \cite{martinet-2007}), and we obtain 
a basis of minimal vectors by again replacing two convenient vectors 
$e_k,e_\ell$ by some vectors $e-e_i$ and $f\mp e_j$. 

\smallskip 

\noi\underbar{Type $(4,2)$}. Here we have $\Lb=\la\Lb',e,f\ra$ 
with $e$ of denominator~$4$ and $f$ of denominator~$2$, and there 
exist minimal vectors $x\in e+\Lb'$ or $x\in e+f+\Lb'$ 
and $y\in f+\Lb'$ or $f+2e+\Lb'$, with $e$ and $f$ as in~\cite[Table~7]{kms-2010}. 
Changing the representatives for $\Lb/\Lb'$ if need be, 
we may assume that $x\in e+\Lb'$ and $y\in f+\Lb'$, and it suffices 
to show as for quotients of type $(3,3)$ that the numerators of $x$
and $y$ have some component equal to~$\pm 1$. 
This is clear for $y$: since $\Lb'''=\la\Lb',f\ra$ has index~$2$, 
the numerator of $y$ has components $\pm 1$ 
whenever those of $f$ are $\pm 1$. The same is obviously true 
with $e$ for the first six rows of~\cite[Table~7]{kms-2010}, since~$e$ itself 
is minimal. To deal with the remaining $13$ rows, we observe that 
$\Lb''=\la\Lb',e\ra$ has index~$4$, and that the components 
of the numerator of $x$ are odd exactly when those of $e$ are. 
By~\cite[Table~7]{kms-2010}, there are $t\ge 5$ such components. 
If none was equal to~$\pm1$, there would be $t$ components $\pm 3$ 
in the numerator of $x$, hence there would exist a pair $L,L'$ of 
lattices with $[L:L']=3$ in dimension $n'=9+1-t\le 5$, 
which would contradict the results of~\cite[Table 11.1]{martinet-2001}. 

\smallskip 

\noi\underbar{Type $(2,2,2)$}. 
This case is dealt with as part~\cite[Lemma~3.1]{martinet-2007}, but proofs
are only sketched there. In the following section we give more details for the part needed here,
by proving the following proposition.

\begin{prop}[\cite{martinet-2007}] \label{prop:maxindex8}
A $2$-elementary lattice of maximal index $\imath = 8$ 
and dimension~$n\le 9$, which is generated by its minimal vectors, 
has always a basis of minimal vectors. 
\end{prop}

\section{Lattices with 2-elementary quotients} 

\label{sec:2elementary}

In this section we give detailed proofs for two assertions, 
which were only sketched in \cite{martinet-2007}, 
concerning elementary quotients of order $4$ and~$8$. 
We first consider lattices of index $\imath\le 4$ 
and dimension~$n\le 10$, with the usual notation 
$\Lb$, $\Lb'$, $\cB=(e_1,\dots,e_n)$, 
assuming that $\Lb$ is generated by its minimal vectors. 
By showing that these lattices always have a basis of minimal
vectors, we obtain a proof for the following proposition,
which is even a bit stronger than the assertion of Theorem~\ref{mainthm} for these indices.

\begin{prop}[\cite{martinet-2007}] \label{prop:smallindex}
A lattice of maximal index $\imath\le 4$ 
and dimension~$n\le 10$, which is generated by its minimal vectors, 
has always a basis of minimal vectors. 
\end{prop}

The proof given in \cite{martinet-2007} 
is only sketched for non-cyclic quotients. 
We complete it in this section. By the $10$-dimensional lattice described in
Section~\ref{sec:counterexample}, it 
turns out that Proposition~\ref{prop:smallindex} is best possible, 
because it has maximal index~$\imath=5$.
Thus our new $10$-dimensional counterexample is minimal 
for both the dimension and the maximal index. 
We note that the $11$-dimensional lattice constructed in~\cite{cs-1995} 
has maximal index $\imath=3$.

\begin{proof}[Proof of Proposition~\ref{prop:smallindex}]
Lattices with maximal index~$\imath=2$, 
and generated by their minimal vectors are easily proved to possess 
a basis of minimal vectors, regardless of the dimension: 
indeed $\Lb$ is generated over $\Lb'$ endowed with a basis 
$\cB=(e_1,\dots,e_n)$ of minimal vectors by one minimal vector 
$x=\frac{a_1 e_1+\dots+a_n e_n}2$ with $\abs{a_j}\le 2$, 
hence with some $a_i$ equal to $\pm 1$; replacing $e_i$ by $x$ 
in $\cB$ yields a basis of minimal vectors for $\Lb$.

When $\Lb/\Lb'$ is cyclic of order $d=3$ or $4$, it is proved in 
\cite{martinet-2007} that $\Lb$ possesses a minimal vector 
of the form $x=\frac{a_1 e_1+\dots+a_n e_n}d$ with at least one $a_i$ 
equal to $\pm 1$, so that replacing $e_i$ by $x$ in $\cB$ 
gives us a basis of minimal vectors. 

Assume now that $\Lb/\Lb'$ is $2$-elementary of order~$4$. 
Then we have $\Lb=\Lb'\cup(e+\Lb')\cup(f+\Lb')\cup(g+\Lb')$ 
where $e,f,g\in\Lb$ have denominators~$2$ and numerators $0$ or~$1$ 
when expressed with respect to the basis $\cB'$ of $\Lb'$, and $e+f+g\in\Lb'$. 

There is a well-defined partition 
of $\{1,\dots,n\}$ into subsets $I_1,I_2,I_3,I_4$ such that 
$$e=\frac 12\sum_{i\in I_1\cup I_2}\,e_i\ \nd \ 
f=\frac 12\sum_{i\in I_2\cup I _3}\,e_i$$ 
(hence $g=\frac 12\sum_{i\in I_1\cup I_3}\,e_i$). 
We set $p_k=\abs{I_k}$, $m_1=p_1+p_2$, $m_2=p_2+p_3$, 
and $m=p_1+p_2+p_3$ (thus $p_4=n-m$). Note that at least 
two of the intervals $I_1,I_2,I_3$ are non-empty. 

Permuting $e,f,g$ if necessary, we may assume that the cosets 
of $e$ and $f$ contain minimal vectors 
$$x=\frac{a_1 e_1+\dots+a_n e_n}2\ \nd\ 
y=\frac{b_1 e_1+\dots+b_n e_n}2\,.$$ 
By Proposition~\ref{propbound}, we have 
$a_i=\pm 1$ if $i\in I_1\cup I_2$ and $a_i=0,\pm 2$ otherwise, 
and similarly 
$b_i=\pm 1$ if $i\in I_2\cup I_3$ and $b_i=0,\pm 2$ otherwise. 

We now construct a basis of minimal vectors for $\Lb$ 
by considering successively two cases. 

\smallskip 

\noi\underbar{Case 1}. Assume first that some $a_i,\,i\in I_3$ 
or some $b_i,\,i\in I_1$ is equal to $\pm 2$. 
Exchanging $x$ and $y$, negating $y$, and permuting some $e_i$ 
if necessary, we may assume that $b_1=2$ and $a_1=1$. 
Then 
$$
4x+2y= 4 e_1 + \sum_{i\ge 2}\, (2a_i + b_i) e_i
$$
and we obtain
$$
x-e_1=\frac{-2y+\sum_{i\ge 2}\,c_i e_i}4
$$ 
where $c_i=2a_i+b_i$ is odd for $i\in I_2\cup I_3$. 
We have thus constructed a sublattice $L'=\langle \Lb', x-e_1\rangle $ of $\Lb=\langle \Lb', e,f\rangle $ 
such that $\Lb/L'$ is cyclic, which implies the existence 
of a minimal basis for $\Lb$ since we assume $n\le 10$. 

\smallskip 

\noi\underbar{Case 2}. Assume now that all $a_i$ for $i\in I_3$ 
and all $b_i$ for $i\in I_1$ are zero. Then we obtain a basis of 
minimal vectors for~$\Lb$ by replacing one $e_i$ by~$x$ and one $e_j$ 
by $y$ where $i,j$ are chosen in two distinct 
intervals $I_k$. 
\end{proof}

\medskip 

We can now also deal with the Type $(2,2,2)$ case at the end of
Section~\ref{seclargeindex}, by proving the proposition 
about elementary quotients $\Lb/\Lb'$ of order~$8$ in dimension $n\le 9$.

\begin{proof}[Proof of Proposition~\ref{prop:maxindex8}]
One easily checks that binary codes of dimension~$3$ and length $\ell\le 9$ 
(and even $\ell\le 10$) are of weight $wt\le 4$. 
Hence if $\Lb/\Lb'$ is of type $(2,2,2)$, one may 
write $\Lb$ in the form $\Lb=\la\Lb',e,f,g\ra$ where $g$ 
is of the form $g=\frac{e_{i_1}+e_{i_2}+e_{i_3}+e_{i_4}}2$, 
and there exist minimal vectors $x\in e+\Lb'$ and $y\in f+\Lb'$. 
We now consider the two cases in the proof of Proposition~\ref{prop:smallindex} 
for the lattice $L=\la\Lb',e,f\ra$. 

If Case~1 holds, there exists $L'\subset L$ 
such that $\Lb/L'$ is of type $(4,2)$, and then the existence 
of a basis of minimal vectors has been proved 
in Section~\ref{seclargeindex}. If Case~2 holds, we first construct 
as above a basis for $L$ by replacing two vectors $e_i,e_j$ by $x,y$, 
and then a basis for $\Lb$ by replacing by $g$ some $e_{i_k}$ 
which belongs to the numerator of $g$ but not to those 
of $e$ and $f$.
\end{proof}

\medskip 

\noi{\bf Remark.} When $[\Lb/\Lb']$ is elementary of order~$4$, 
the existence of a basis of minimal vectors for $\Lb$ 
could have been proved {\em without assuming the condition $n\le 10$}, 
as asserted in \cite{martinet-2007}, but with a less simple proof.

\section{Lattices of maximal index 6} \label{sec:index6} 

We may write $\Lb=\la\Lb',e\ra$ 
where $e=\frac{\sum_{i=1}^9\,a_i e_i}6$ with increasing 
$a_i\in\{0,1,2,3\}$ and $(e_1,\dots,e_9)$ is a basis for~$\Lb'$. 
For $i=0,1,2,3$, we denote by $m_i$ the number of $a_j$ equal to~$i$ 
and set $m=m_1+m_2+m_3$. We have $m=8$ 
(six systems $(m_1,m_2,m_3)$ listed in~\cite[Table~11.1]{martinet-2001}) 
or $m=9$ 
($20$ systems listed in~\cite[Table~2]{kms-2010}). 
We assume that $\Lb$ is generated by its minimal vectors, 
which amounts to the existence of either a minimal vector 
$x=\frac{\sum_{i=1}^9\,b_i e_i}6\in e+\Lb'$ 
(then, $b_i\equiv a_i\mod 6$ and $\abs{b_i}\le 6$), 
or the existence of two minimal vectors 
$y=\frac{\sum_{i=1}^9\,c_i e_i}3\in 2e+\Lb'$ 
and 
$z=\frac{\sum_{i=1}^9\,d_i e_i}2\in 3e+\Lb'$, 
with $c_i\equiv a_i\mod 3$ and $d_i\equiv a_i\mod 2$. 

\smallskip 

Let us first consider the case when there is a minimal vector 
$x\in e+\Lb'$, with, say, $b_i=1$ or $-5$ for $i\le m_1$, 
$b_i=2$ or $-4$ for $m_1<i\le m_1+m_2$, $b_i=\pm 3$ 
if $m_1+m_2<i\le m$ (and we then may assume that $b_i=+3$). 
If $b_i=1$ for some~$i$, then a basis of minimal vectors 
trivially exists. Otherwise, the argument used in the section above 
to deal with denominator~$7$ will show the existence 
of a pair $(L,L')$ of lattices with $\imath=5$ in dimension $10-m_1$, 
which is possible only if $m_1\le 2$ and leaves us 
with the three systems $(2,5,2)$, $(2,4,3)$, and $(1,5,4)$. 

In the first case, $z=\frac{e_1+e_2+e_8+e_9}2$ is minimal, 
and we have 

\smallskip 
\ctl{$x=\frac{4(z-e_1-e_2)+b'_3 e_3+\dots+b'_9 e_9}3$} 

\smallskip\noi 
for some $b'_i$ equal to $1$ or $-2$, 
among which at least three must be odd. 
Replacing $e_9$ by $z$ in $(e_1,\dots,e_9)$  
and one $e_i$ with $3\le i\le 7$ by $x$, 
we obtain a basis of minimal vectors for~$\Lb$. 

The same argument works in the third case, 
taking $z=\frac{e_1+e_7+e_8+e_9}2$. 

In the second case, $y=\frac{e_1+e_2+e_3-e_4-e_5+2e_6}3$ is minimal, 
we have 

\smallskip 
\ctl{$x=\frac{4y\pm e_7\pm e_8\pm e_9+b'_3 e_3+\dots+b'_6 e_6}2$} 

\smallskip\noi 
(with $b'_i=0,\pm 1$ or $3$), and we obtain a basis of minimal 
vectors for~$\Lb$ by replacing $e_1$ by $x$ and $e_9$ by~$y$. 

\smallskip 

From now on we assume that $e+\Lb'$ does not contain any minimal 
vector, and work with the minimal vectors $y\in 2e+\Lb'$ 
and $z\in 3e+\Lb'$. Proposition~\ref{propbound}
shows that we have $\abs{c_i}\le 3$ and $\abs{d_j}\le 2$. 
By the results of \cite{martinet-2001} and \cite{kms-2010}
for $m=8$ (resp. $9$), there exist $6$ 
(resp. $20$ minimal classes), 
among which $5$, those with $s>14$ (resp. $4$, those with $s>17$) 
contain minimal vectors in $e+\Lb'$. 
So we are left with~$1$ (resp. $15$) minimal classes. 

\smallskip 

We can get rid of the remaining minimal class with $m=8$ 
and of the two minimal classes with $m=9$ and $s=17$ by the following 
argument, which is valid for any dimension: 
assume that we have $m_1+m_2=6$ and that $3e+\Lb$ contains 
some minimal vector~$z$; then a vector of the form 
$y=\frac{\pm e_1\pm\dots\pm e_5\pm 2e_6}3$ is minimal, 
and since the numerator of $z$ has component $d_m=\pm 1$, 
$(y,e_2,\dots,e_{m-1},z,\dots)$ is a basis of minimal vectors for~$\Lb$. 

We could get rid in a similar way 
(by a slightly more complicated argument) of the four minimal classes with
$m=9$ and $s=15$, but it appears to be very difficult to construct 
bases of minimal vectors by elementary arguments 
for the $9$ minimal classes with $m=9$, for which $s(\Lb)=s(\Lb')=9$ is possible. 
So in this case we used the general approach described in Section~\ref{sec:algorithmic_approach}
to exclude the existence of lattices that are generated by minimal vectors,
but which do not have a basis among them.

\smallskip 

In fact, we ran a computer calculation on all~$2574$ 
possible cases with additional minimal vectors $y\in 2e+\Lb'$ and $z\in 3e+\Lb'$, 
having coefficients $|c_i|\leq 3$ and $|d_j|\leq 2$,
and falling into one of the~$15$ cases with $d=6$ and $s\leq 17$ listed in~\cite[Table~2]{kms-2010}.
Using the general approach of Section~\ref{sec:algorithmic_approach},
these computations show that all but~$22$ of these cases are {\em infeasible},
that is, lattices respectively Gram matrices with these parameters do
not exist.
All of the $22$~feasible cases turn out to have parameters $m_1=5$ and
$m_2=4$.
Considering the corresponding Gram matrices 
displayed in the file {\tt Gramindex.gp} in the online appendix of~\cite{kms-2010} 
shows that in all of these cases, 
the set of minimal vectors also contains a basis of minimal vectors.
By this we complete the proof of Theorem~\ref{mainthm} for lattices of index~$6$.

\section{Lattices of maximal index 5} \label{sec:index5} 

We consider lattices of dimension~$n$ ($n$ will be~$8$,~$9$ or~$10$) 
of the form 
$\Lb= \langle \Lb' , e, e' \rangle$ where
$$e=\frac{e_1+\dots+e_{m_1}+2(e_{m_1+1}+\dots+e_{m_1+m_2})}5\,,$$ 
$$e'=\frac{2(e_1+\dots+e_{m_1})-(e_{m_1+1}+\dots+e_{m_1+m_2})}5\,,$$ 
and $\cB=(e_1,\dots,e_n)$ is a basis for $\Lb'$. 
From \cite{martinet-2001}, we know that $\ell:=m_1+m_2$ is at least $8$, 
that if $\ell=8$ (resp. $\ell=9$) we must have $2\le m_1\le 6$ 
(resp. $1\le m_1\le 9$), and that lattices with 
$(m_1,m_2)=(2,6)$, $(4,4)$, $(6,2)$, $(1,8)$ or $(8,1)$ 
necessarily have bases of minimal vectors. 
In the remaining cases, we must use the hypothesis 
that~$\Lb$ is generated by its minimal vectors, which amounts to saying 
that there exists some minimal vector in one of the two cosets 
$e+\Lb'$ or $2e+\Lb'$ modulo~$\Lb'$, and since we may exchange 
$e$ and $e'$ (see \cite{martinet-2001}, Example~3.3), we may and shall assume 
that the coset of $e$ contains some minimal vector~$x$. 
This vector must be of the form 
$$x=\frac 15\,\sum_{i=1}^n\,a_i\,e_i$$ 
where the $a_i$ satisfy the following congruences modulo~$5$: 
$a_i\equiv 1$ if $i\le m_1$, $a_1\equiv 2$ if $m_1<i\le\ell$, 
and $a_i\equiv 0$ if $i>\ell$, and are bounded from above by~$5$ 
(because $i(\Lb)=5$). 

\smallskip\noi 
Assuming that $n\le 9$, we shall now derive a contradiction 
from the assumptions 
\newline (1) there exists a vector $x$ as above, and 
\newline (2) $\Lb$ has no basis of minimal vectors. 

First observe that $a_i=1$ is impossible, since replacing 
$e_i$ by $x$ yields a basis of minimal vectors for~$\Lb$. 
We may thus assume that $a_i=-4$ for $1\le i\le m_1$. 

To simplify the notation, for $i=4,2,3,5$, set 
$$\Sigma_i=\sum_{\abs{a_k}=i} e_k\ \ \nd\ \ 
m'_i=\abs{\{k\mid{\abs{a_k}=i}\}}\,;$$ 
we thus have $m'_4=m_1$, $m'_2+m'_3=m_2$, $m'_5=n-\ell$, and 
$$x=\frac{-4\Sigma_4+2\Sigma_2-3\Sigma_3+5\Sigma_5}5\ \ \nd\ \ 
e=\frac{\Sigma_4+2(\Sigma_2+\Sigma_3)}5\,.\eqno{(*)}$$ 
We now write down two identities which will allow us 
to make use of inequalities involving denominators 
$4$ and $2$ first, and then~$3$: 
$$ 
\aligned 
x+\Sigma_4+\Sigma_3-\Sigma_5&= 
\frac{(-x+\Sigma_3+\Sigma_5)+2\Sigma_2}4\\ 
2x+\Sigma_4+\Sigma_3-\Sigma_2-2\Sigma_5&= 
\frac{x-\Sigma_4-\Sigma_2-\Sigma_5}3 
\,.\endaligned$$ 
From the classification of lattices of maximal index $2$ and~$4$ 
(resp.~$3$); see \cite[Theorem~2.2 and Table~11.1]{martinet-2001})
, 
we deduce the inequalities 

\smallskip 

\ctl{$m'_3+m'_5\ge 3$, $m'_3+m'_5+2m'_2\ge 7$ and $m'_3+m'_5+m'_2\ge 6$} 

\smallskip 

\ctl{(resp. $m'_4+m'_2+m'_5\ge 5$)\,.} 

\smallskip 

If $m'_3+m'_5=3$, then $f=\frac{-x+\Sigma_3+\Sigma_5}2$ 
belongs to $\Lb$, and a short calculation shows that $f=2\,e-\Sigma_2$. 
Hence replacing in $\cB$ an $e_i$ with $a_i=-3$ or $5$ by $f$, 
we obtain a lattice containing $\Lb'$ and 

\smallskip 

\ctl{$e=6e-5e=3(f+\Sigma_2)-(\Sigma_4+2\Sigma_2+2\Sigma_3)$\,,} 

\noi 
hence the lattice~$\Lb$, which shows that in this case, $\Lb$ possesses 
a basis of minimal vectors. 

If $m'_3+m'_5+2m'_2=7$, let 
$f=\frac{(-x+\Sigma_3+\Sigma_5)+2\Sigma_2}4$, and let $y$ 
be a vector $e_i$ with $a_i=-3$, $2$ or $5$. Then $f-y$ is minimal, 
and a short calculation shows that $f=e$. Hence replacing in $\cB$ 
a convenient $e_i\ne y$ by $f-y$, we again obtain~$\Lb$. 

If $m'_4+m'_2+m'_5=5$, let 
$g=\frac{x-\Sigma_4-\Sigma_2-\Sigma_5}3$, and let $y$ 
be a vector $e_i$ with $a_i=-4$, $2$ or $5$. 
Then $g+y$ is minimal, and we have this time $2g=-e+\Sigma_4$, 
which again shows the existence in this case of a basis of minimal 
vectors for~$\Lb$. 

Summarizing, we have: 

\begin{lemma}\label{leminv} 
Let $\Lb$ be a lattice of maximal index $5$ generated by its minimal 
vectors but having no basis of minimal vectors. Then $\Lb$ is 
generated by a basis $\cB=(e_1,\dots,e_n)$ of minimal vectors 
for a lattice $\Lb'$ of index~$5$ in $\Lb$ and a minimal vector $x$ 
as in $(*)$ which satisfies the conditions 
$$m'_3+m'_5\ge 4,\quad m'_3+m'_5+2m'_2\ge 8\ \,\nd\,\ 
m'_4+m'_2+m'_5\ge 6\,.$$ 
\end{lemma} 

\begin{corol} \
Let $\Lb$ be a lattice of dimension $\le 9$ and maximal index $5$ 
generated by its minimal vectors but having no basis of minimal 
vectors. Then $\Lb$ has the invariants $m_1=3$, $m_2=5$ and has 
a minimal vector $x$ with invariants $m'_2=2$ and $m'_3=3$. 
\end{corol} 

\begin{proof} 
We know that $n=8$ is impossible. Adding the first and the third 
inequality in Proposition~\ref{leminv}, we get $\ell\ge 10-2m'_5$. 
Hence we must have $n=9$ and $m'_5=1$, thus $\ell=8$. 
Adding the last two inequalities, we get $\ell+2m'_2\ge 14-2m'_5$, 
hence $m'_2\ge 2$. From $m'_3\ge 4-m'_5=3$, we get $m_2\ge 5$, 
and since $m_2=6$ is excluded, we are left with $m_2=5$, 
which implies $m_1=3$, $m'_2=2$ and $m'_3=3$. 
\end{proof} 

To prove the existence of a basis of minimal vectors for lattices 
of index~$\imath=5$, it now suffices to consider the case 
when we may write 
\begin{equation} \label{eqn:edef}
e =\frac{e_1+e_2+e_3+2(e_4+e_5+e_6+e_7+e_8)}5\ \quad \nd
\end{equation} 
\begin{equation} \label{eqn:xdef}
x =\frac{-4(e_1+e_2+e_3)+2(e_4+e_5)-3(e_6+e_7+e_8)+5e_9}5 
\,.
\end{equation} 
Here the consideration of denominators $2$, $3$ or $4$ as above does not 
produce obvious new minimal vectors. 
We therefore use 
the general algorithmic approach of Section~\ref{sec:algorithmic_approach}
to deal with this case.
As already described there, choosing $B=(e,e_2,\ldots, e_9)$ as a basis,
we obtain coordinates $\bar{e}^{(1)} =(5,-1,-1,-2,-2,-2,-2,-2,0)$ for~$e_1$ and $a^{(1)}=(-4,0,0,2,2,-1,-1,-1,-1)$ for~$x$.
Assuming a minimum of~$1$, we get ten linear conditions $G[\bar{e}^{(i)}]=1$, $i=1,\ldots,9$, $G[a^{(1)}]=1$
for Gram matrices~$G$ on the Ryshkov polyhedron (the set of Gram matrices in $\cS^9$ with $G[z]\geq 1$ for all $z\in\Z^9$). 
We can make use of some symmetry, as the set of coordinate vectors~$\bar{e}^{(i)}$ and~$a^{(1)}$ 
is invariant with respect to permutations of~$e_2,e_3$, of~$e_4,e_5$ and of~$e_6,e_7,e_8$. This yields four additional
linear conditions, and we obtain a polytope with $25$~vertices satisfying all of the prescribed equations.
Its vertex barycenter scaled by~$900$ is the Gram matrix
\begin{equation} \label{eqn:n9d5gram}
\begin{pmatrix}
1104 & 54 & 54 & 552 & 552 & 528 & 528 & 528 & 312 \\
54 & 900 & 66 & 27 & 27 & -142 & -142 & -142 & 267 \\
54 & 66 & 900 & 27 & 27 & -142 & -142 & -142 & 267 \\
552 & 27 & 27 & 900 & 138 & 102 & 102 & 102 & -87 \\
552 & 27 & 27 & 138 & 900 & 102 & 102 & 102 & -87 \\
528 & -142 & -142 & 102 & 102 & 900 & 216 & 216 & 186 \\
528 & -142 & -142 & 102 & 102 & 216 & 900 & 216 & 186 \\
528 & -142 & -142 & 102 & 102 & 216 & 216 & 900 & 186 \\
312 & 267 & 267 &-87 & -87 & 186 & 186 & 186 & 900 
%
\end{pmatrix}
.
\end{equation}
From it we obtain a list of ten additional 
coordinate vectors
$a^{(j)}$ which are minimal (satisfying $G[a^{(j)}]=1$):
\begin{eqnarray*}
( -1, 0, 0, 1, 1, 0, 0, 0, 1), &
( -1, 0, 0, 1, 0, 0, 0, 0, 0), &
( -1, 0, 0, 0, 1, 0, 0, 0, 0),\\
( -2, 1, 0, 1, 1, 1, 1, 1, 0), &
( -2, 0, 1, 1, 1, 1, 1, 1, 0), &
( -2, 0, 0, 1, 1, 1, 1, 0, 0), \\
( -2, 0, 0, 1, 1, 1, 0, 1, 0), &
( -2, 0, 0, 1, 1, 0, 1, 1, 0), &
( -3, 1, 1, 1, 1, 1, 1, 1, 0), \\
( -3, 0, 0, 1, 1, 1, 1, 1, 1). 
\end{eqnarray*}

So all $9$-dimensional 
lattices generated by minimal vectors of type $e_i$ and $x$ as above,
have ten additional pairs of minimal vectors, and among them we 
find always a basis of minimal vectors. Take for example~$(e-e_4,e_2,\dots,e_9)$.
We have thus proved that $9$-dimensional lattices of maximal index 
$\imath=5$, and which are generated by their minimal vectors, 
indeed have bases of minimal vectors.

Here is a ``check'' for the obtained result:
The $20$~linear conditions of minimal vectors~$S$ have rank~$19$.
Up to scaling, there is therefore a unique {\em perfection relation} (cf.~\cite{bm-2009})
between the $20$~orthogonal projections $p_y$ (in direction $y$) associated to these vectors. 
Setting $S=S_1\cup S_2$ with $S_1=\{\pm e_1,\dots,\pm e_9,\pm x\}$, 
we find that this relation has the simple form 
\begin{equation}  \label{eqn:projection_relation}
\sum_{y\in S_1/\{\pm\}}\,p_y=\sum_{y\in S_2/\{\pm\}}\,p_y\,.
\end{equation} 
By~\cite[Lemma~2.9]{bm-2009}, this implies the identity 
$$\sum_{y\in S_1/\{\pm\}}\,N(y)=\sum_{y\in S_2/\{\pm\}}\,N(y)$$ 
between norms.  
This implies (still assuming that $\min\Lb=N(e_i)$) 
that all vectors in $S_2$ are actually minimal, 
so that we obtain a simple proof using calculations ``by hand'' only.
Note however, that guessing the necessary identity~\eqref{eqn:projection_relation} 
would have been difficult without the help of a computer!

\section{A $10$-dimensional counterexample} \label{sec:counterexample}

To complete the proof 
of Theorem~\ref{mainthm} for $\imath=5$, 
it only remains to exhibit a $10$-dimensional counterexample. 
With the notations of the previous section, we give one
with parameters $m_1=3$, $m_2=7$ and $m'_2=3$, $m'_3=4$.
With respect to the basis $(e,e_2,\ldots, e_{10})$
we consider coordinates $\bar{e}^{(1)}=(-5,1,1,2,2,2,2,2,2,2)$ for~$e_1$ 
and $a=(-4,0,0,2,2,2,1,1,1,1)$ for an additional minimal vector~$x$
that generates the lattice $\Lb$ together with the $e_i$.
We can make use of some symmetry, as the set of coordinate vectors ($\bar{e}^{(i)}$, $i=1,\ldots, 10$, and~$a$) 
is invariant with respect to permutations of~$e_2,e_3$, of~$e_4,e_5,e_6$ and of~$e_7,e_8,e_9,e_{10}$. 
Assuming a fixed minimum, all of these linear conditions turn out to 
define a polytope of Gram matrices with $154$~vertices. Its vertex barycenter is a
Gram matrix which we can scale to have integral coordinates and minimum $6209280$.
Up to isometry, it defines a lattice $\Lb$ having only the $s=11$ pairs 
of minimal vectors that we assumed from the beginning (namely $x$ and
the ten vectors~$e_i$).
It is readily verified that any system of $10$~independent vectors 
extracted from this set generates a sublattice $L$ of~$\Lb$ 
with $\Lb/L$ cyclic of order $2$, $3$, $4$ or~$5$, but not~$1$.

Since the vertex barycenter of the ``realization polytope'' mentioned above has quite 
inconvenient coordinates, we provide below a slightly nicer counterexample in the same polytope.
An analysis of the $154$~vertices shows that it is possible
to take some of the midpoints of two vertices as an interior point.
All of these counterexamples have the same parameters.
Among them, there is a unique one that we can scale to have integral 
coordinates and minimum~$48$:
$$
\begin{pmatrix}
88 & -3 & -3 & 40 & 40 & 40 & 26 & 26 & 26 & 26 \\
-3 & 48 & 10 & 5 & 5 & 5 & -13 & -13 & -13 & -13 \\
-3 & 10 & 48 & 5 & 5 & 5 & -13 & -13 & -13 & -13 \\
40 & 5 & 5 & 48 & 14 & 14 & 4 & 4 & 4 & 4 \\
40 & 5 & 5 & 14 & 48 & 14 & 4 & 4 & 4 & 4 \\
40 & 5 & 5 & 14 & 14 & 48 & 4 & 4 & 4 & 4 \\
26 & -13 & -13 & 4 & 4 & 4 & 48 & 8 & 8 & 8 \\
26 & -13 & -13 & 4 & 4 & 4 & 8 & 48 & 8 & 8 \\
26 & -13 & -13 & 4 & 4 & 4 & 8 & 8 & 48 & 8 \\
26 & -13 & -13 & 4 & 4 & 4 & 8 & 8 & 8 & 48 \\
\end{pmatrix}
$$

%
%

\section*{Acknowledgments} 

We would like to thank the anonymous referee for his helpful comments.
The second author would like to thank the 
Institut de Math\'ematiques at Universit\'e Bordeaux~1 
for its great hospitality during two visits, 
on which major parts of this work were created.

%
%


\providecommand{\bysame}{\leavevmode\hbox to3em{\hrulefill}\thinspace}
\providecommand{\MR}{\relax\ifhmode\unskip\space\fi MR }
\providecommand{\MRhref}[2]{%
\href{http://www.ams.org/mathscinet-getitem?mr=#1}{#2}
}
\providecommand{\href}[2]{#2}

\end{document}